\documentclass[a4paper,12pt]{extarticle}

\usepackage[utf8]{inputenc}
\usepackage{geometry}
\usepackage{graphicx}
\usepackage{booktabs}
\usepackage{mathrsfs}
\usepackage{bm}
\usepackage{xcolor}
\usepackage{algorithm}
\usepackage{algpseudocode}
\usepackage{mathtools}
\usepackage{enumitem}
\usepackage{tikz-cd}

\usepackage{newtxtext}
\usepackage{newtxmath}
\let\Bbbk\relax

\usepackage{amssymb}

\usepackage{amsthm}

\usepackage{hyperref}
\usepackage{xcolor}

\newtheorem{theorem}{Theorem}[section]
\newtheorem{proposition}[theorem]{Proposition}
\newtheorem{lemma}[theorem]{Lemma}
\newtheorem{corollary}[theorem]{Corollary}
\newtheorem{definition}[theorem]{Definition}

\newcommand{\Sym}{\operatorname{Sym}}

\newcommand{\g}{\mathfrak{g}}
\newcommand{\Ad}{\operatorname{Ad}}

\newcommand{\ch}{\operatorname{ch}}
\newcommand{\td}{\operatorname{td}}

\theoremstyle{definition}
\newtheorem{example}[theorem]{Example}

\begin{document}

\pagestyle{plain}

\title{Mirror Duality in a Spencer-Type Complex: Analytic and Riemann-Roch Perspectives}
\author{Dongzhe Zheng}
\date{}
\maketitle

\begin{abstract}
We introduce and analyze a Spencer-type elliptic complex on the space of differential forms valued in symmetric powers of an adjoint bundle, $\Omega^\bullet(X)\otimes \mathrm{Sym}^\bullet(G)$. The complex is governed by a total differential $D_{\lambda,\psi}$ depending on a section $\psi\in\Gamma(G)$ and a real parameter $\lambda$. The central result of this paper is an algebraic realization of mirror-type duality and parameter robustness at the \emph{chain-level}. We demonstrate that sign flips ($\lambda\mapsto -\lambda$ or $\psi\mapsto -\psi$) and rescaling ($\lambda\mapsto \alpha\lambda$) of the deformation parameters correspond to simple conjugations of the differential $D_{\lambda,\psi}$ by elementary zero-order automorphisms. This provides a unified, conceptual foundation for the invariance of topological invariants that is often established via case-by-case analytic methods. Analytically, this framework implies the invariance of harmonic space dimensions under the mirror map $\psi\mapsto -\psi$. Algebraically, the Grothendieck--Riemann--Roch index formula for the complex's hypercohomology is shown to be manifestly independent of $(\lambda, \psi)$, determined solely by the characteristic classes of a universal virtual bundle. The theory is fully compatible with equivariant localization and is verified with concrete applications on Calabi--Yau backgrounds, including K3 surfaces and elliptic curves. This framework thus offers a rigorous, chain-level explanation for the parameter robustness intrinsic to Witten-type deformations and localization phenomena, grounding them in a fundamental algebraic conjugation principle.
\end{abstract}

\tableofcontents

\section{Introduction}

In the context of twisted de Rham complexes, Witten deformations, and equivariant localization, the choice of \emph{sign and scale of deformation parameters} frequently arises yet should be inconsequential: for instance, the cohomology and index of Morse/Witten deformations are independent of the sign and scaling of parameters \cite{Witten1982}, and equivariant indices and localization formulas are likewise insensitive to such choices \cite{AtiyahBott1984,BerlineVergne1982,DuistermaatHeckman1982,AtiyahSinger1968}. However, this "insensitivity" is often handled case-by-case through analytic methods. The goal of this paper is to realize this robustness \emph{at the chain level} as zero-order conjugation on a natural Spencer-type de Rham--Koszul bicomplex, and to interface consistently with Riemann--Roch and equivariant localization \emph{at the index level}.

Let $X$ be a closed manifold, $P\to X$ a principal $K$-bundle, and $G=\mathrm{ad}\,P$. We consider the bigraded structure
\[
B_{p,q}=\Omega^p(X)\otimes \mathrm{Sym}^q(G),\qquad S^\bullet=\bigoplus_{p+q=\bullet}B_{p,q},
\]
together with the fiberwise derivation $\iota^{\wedge}_\psi:\mathrm{Sym}^q(G)\to \mathrm{Sym}^{q+1}(G)$ for $\psi\in\Gamma(G)$, satisfying $(\iota^{\wedge}_\psi)^2=0$ and supercommuting with $d$. For $\lambda\in\mathbb{R}$ we define
\[
D_{\lambda,\psi}=d\otimes 1+(-1)^{\deg_{\mathrm{dR}}}\lambda\,(1\otimes \iota^{\wedge}_\psi).
\]
This yields a Spencer-type total complex that connects the de Rham and (fiberwise) Koszul directions \cite{Spencer1969}, and is compatible with the fiber Koszul direction in equivariant Cartan/Weil/BRST models \cite{GuilleminSternberg1999,Quillen1985}.

The first main conclusion of this paper is the chain-level realization of \emph{sign duality and scale normalization}.

\medskip\noindent\textbf{Theorem A (Sign duality and scaling).}
\emph{Let $U|_{B_{p,q}}=(-1)^q\mathrm{id}$ and $J_\alpha|_{B_{p,q}}=\alpha^q\mathrm{id}$ ($\alpha>0$). Then
\[
U D_{\lambda,\psi}U^{-1}=D_{-\lambda,\psi}=D_{\lambda,-\psi},\qquad
J_\alpha D_{\lambda,\psi}J_\alpha^{-1}=D_{\alpha\lambda,\psi}.
\]
In particular, for all $\lambda>0$, $(S^\bullet,D_{\lambda,\psi})$ is chain homotopy equivalent to $(S^\bullet,D_{1,\psi})$ and to $(S^\bullet,D_{1,-\psi})$.}

\medskip

The principal symbol of $D_{\lambda,\psi}$ coincides with the de Rham symbol, so $(S^\bullet,D_{\lambda,\psi})$ is an elliptic complex with self-adjoint Hodge Laplacian and discrete spectrum \cite{AtiyahSinger1968,BottTu1982}. We introduce two types of natural weighted $L^2$ metrics (based on $\|\psi\|$ and $F_\nabla$), prove that under these metrics $U$ is unitary isometric and $J_\alpha$ is bounded invertible, and show that
\[
D_{\lambda,-\psi}-D_{\lambda,\psi}
\]
is a zero-order bounded perturbation that is compact relative to any second-order elliptic operator; the Laplacian satisfies exact perturbation identities, so harmonic space dimensions are preserved under the mirror $\psi\mapsto -\psi$ (Rellich--Kondrachov compact embedding, see \cite{AdamsFournier2003}).

In the algebraic geometric setting, we view $(S^\bullet,D_{\lambda,\psi})$ as a bounded complex of algebraic vector bundles, whose hypercohomology Euler characteristic is computed by Grothendieck--Riemann--Roch \cite{Fulton1998,Hartshorne1977}.

\medskip\noindent\textbf{Theorem B (Spencer--Riemann--Roch).}
\emph{Let $X$ be a smooth projective variety and $G=\mathrm{ad}\,P$ an algebraic vector bundle. Then
\[
\chi\big(S^\bullet,D_{\lambda,\psi}\big)
=\int_X \mathrm{ch}\!\left(\bigoplus_{p,q\ge 0}(-1)^{p+q}\,\Omega_X^p\otimes \mathrm{Sym}^q(G)\right)\cdot \mathrm{td}(X),
\]
independent of $(\lambda,\psi)$. In the equivariant case, this formula holds in equivariant cohomology and can be computed fixed-point-wise via Atiyah--Bott/Berline--Vergne localization \cite{AtiyahBott1984,BerlineVergne1982}.}
\medskip

Filtration by symmetric degree $q$ yields a first-quadrant spectral sequence whose $E_1$ page is given by tensor products of fiber $\iota^{\wedge}_\psi$-cohomology sheaves with cohomology of $\Omega^p$, and all pages are naturally isomorphic under $(\lambda,\psi)\mapsto(\alpha\lambda,\pm\psi)$ \cite{Weibel1994}. In the $c_1(X)=0$ (Calabi--Yau) background, the Todd class simplifies significantly, the index admits a term-by-term decomposition that is term-by-term mirror-invariant; in K3 and elliptic curve cases, this framework provides concrete closed formulas and numerical verification pathways. The correspondence between analytic and algebraic cases is guaranteed by GAGA \cite{Serre1956}.

Methodologically, ellipticity comes from principal symbol computation; spectral stability comes from zero-order compact perturbations and algebraic expansion of the Hodge Laplacian; the index level encapsulates all parameters via the virtual bundle
\[
E_{\mathrm{vir}}=\bigoplus_{p,q\ge 0}(-1)^{p+q}\,\Omega_X^p\otimes \mathrm{Sym}^q(G)
\]
thus decoupling from $(\lambda,\psi)$; the equivariant version is supported by standard localization theory \cite{AtiyahBott1984,BerlineVergne1982}. This paper also abstracts a general "algebraic conjugation principle": for any first-quadrant bicomplex $(\mathcal{B}_{p,q},\delta,\kappa)$ with $[\delta,\kappa]=0$ and $\kappa^2=0$, the operator $\mathbb{D}_\lambda=\delta+(-1)^p\lambda\kappa$ satisfies sign and scale conjugation laws completely parallel to the above, thus possessing the same robustness at the chain/index/equivariant levels. This is naturally compatible with BRST/Koszul, Cartan/Weil, and superconnection perspectives \cite{GuilleminSternberg1999,Quillen1985}.

The differential operator $D_t = d + t \cdot (dW \wedge \cdot)$ studied by \cite{gross2017towards} constitutes a special case of our abstract conjugation principle, and consequently our sign "flipping/scale normalization" chain-level symmetries apply equally to their twisted complexes. In the Riemann surface setting, our framework is fully compatible with their Hodge duality phenomena. The novelty of the present work lies in several key advances: first, we elevate conjugation to a universal chain-level paradigm that encompasses Spencer-type bicomplexes and beyond; second, we provide a complete operator-theoretic analysis of principal symbols, ellipticity, and Hodge Laplacians under appropriate functional analytic completions; third, we unify Riemann-Roch and equivariant localization computations through the systematic use of the virtual bundle $E_{\mathrm{vir}}$, which encapsulates all parameter dependence; and fourth, we introduce the systematic study of adjoint bundle symmetric powers together with spectral sequence organizational algorithms that render the entire framework computationally tractable. This comprehensive approach not only recovers and extends classical results but also provides new computational tools for explicit calculations in concrete geometric settings.

\section{Fundamental Theory: Compatible Pair Spencer Framework}
\label{sec:foundations}

To analyze the mirror-type sign symmetry and parameter robustness phenomena, we build the foundational setup in de Rham–Koszul language.

Let $M$ be a closed connected smooth manifold, $G$ a compact connected semisimple Lie group, and $P \to M$ a principal $G$-bundle with smooth connection $\omega \in \Omega^1(P,\mathfrak{g})$. The (negative-definite) Killing form $\langle \cdot,\cdot\rangle$ identifies $\mathfrak{g} \cong \mathfrak{g}^\ast$.

\subsection{Compatible pairs}

\begin{definition}[Compatible pairs]\label{def:compatible-pair}
Write $V_p \subset T_pP$ for the vertical subspace at $p \in P$. A compatible pair $(D,\lambda)$ consists of:
\begin{itemize}
\item A $G$-invariant distribution $D \subset TP$ satisfying the strong transversality condition
\[
D_p \oplus V_p = T_pP \quad \text{for all } p \in P.
\]
\item A $G$-equivariant section $\lambda \in \Gamma(\operatorname{ad}^\ast P) \cong \Gamma(\operatorname{ad}P)$ that is covariantly constant for the connection induced on $\operatorname{ad}^\ast P$:
\[
\nabla^{\operatorname{ad}^\ast}\lambda = 0 \quad \text{(equivalently on $P$: } d\lambda + \operatorname{ad}^\ast_\omega \lambda = 0 \text{)}.
\]
\item The compatibility condition $D = \ker \alpha$ for the basic 1-form $\alpha := \langle \lambda, \omega \rangle$, i.e.
\[
D_p = \{ v \in T_pP \,:\, \langle \lambda(p), \omega(v) \rangle = 0\}.
\]
\end{itemize}
\end{definition}

The constructions below use only $(P,\omega,\lambda)$; the distribution $D$ provides a geometric origin for $\lambda$ via $\alpha=\langle\lambda,\omega\rangle$ but does not enter the analytic statements.

\subsection{Spencer-type de Rham–Koszul complex}

Let $\mathcal{G}:=\operatorname{ad}P$ be the adjoint bundle. Consider the bigraded vector bundle
\[
\mathcal{B}^{p,q} := \Omega^p(M) \otimes \operatorname{Sym}^q(\mathcal{G}), \qquad p,q \ge 0,
\]
with total degree $N=p+q$ and totalization
\[
\mathcal{S}^N := \bigoplus_{p+q=N} \mathcal{B}^{p,q}, \qquad \mathcal{S}^\bullet := \bigoplus_{N\ge 0} \mathcal{S}^N.
\]

\begin{definition}[Spencer-type differential]\label{def:spencer-complex}
Define a degree $+1$ operator $\mathcal{D}_{\lambda} : \mathcal{S}^\bullet \to \mathcal{S}^{\bullet+1}$ by
\[
\mathcal{D}_{\lambda}\big(\omega \otimes s\big) \;:=\; d\omega \otimes s \;+\; (-1)^p\, \omega \otimes \delta^{\lambda}(s), \quad \omega \in \Omega^p(M),\ s \in \operatorname{Sym}^q(\mathcal{G}),
\]
where $\delta^\lambda$ is the graded derivation of degree $+1$ on $\operatorname{Sym}(\mathcal{G})$ uniquely determined by:
\begin{align*}
&\text{(A) For $v \in \mathcal{G}_x \cong \mathfrak{g}$ and $w_1,w_2 \in \mathcal{G}_x$:}\\
&\hspace{3em}\big(\delta^{\lambda}(v)\big)(w_1, w_2) := \tfrac{1}{2}\Big(\langle\lambda(x), [w_1,[w_2,v]]\rangle + \langle\lambda(x), [w_2,[w_1,v]]\rangle\Big),\\
&\text{(B) Graded Leibniz: for homogeneous $s_1 \in \operatorname{Sym}^p, s_2 \in \operatorname{Sym}^q$:}\\
&\hspace{3em}\delta^{\lambda}(s_1 \odot s_2) := \delta^{\lambda}(s_1)\odot s_2 + (-1)^p\, s_1 \odot \delta^{\lambda}(s_2).
\end{align*}
The Spencer cohomology is $H^\bullet_{\mathrm{Spencer}}(D,\lambda):=H^\bullet(\mathcal{S}^\bullet,\mathcal{D}_\lambda)$.
\end{definition}

\begin{lemma}[Square-zero, linearity, and commutation]\label{lem:delta-square-zero}
For the operator $\delta^\lambda$ above: $(\delta^\lambda)^2=0$; $\delta^{-\lambda}=-\,\delta^\lambda$; and $[\,d\otimes 1,\,1\otimes \delta^\lambda\,]=0$. Consequently $\mathcal{D}_{\lambda}^2=0$.
\end{lemma}

\begin{proof}
Fiberwise, for fixed $x$ and $v \in \mathcal{G}_x$, set $Q_v(w):=\langle\lambda(x),[w,[w,v]]\rangle$, a quadratic polynomial in $w$. Then $\delta^\lambda(v)$ is the polarization (Hessian) of $Q_v$, so applying $\delta^\lambda$ twice corresponds to a third derivative of a quadratic form, which vanishes. Linearity in $\lambda$ is immediate. Since $d$ acts on $\Omega^\bullet(M)$ and $\delta^\lambda$ on $\operatorname{Sym}(\mathcal{G})$, they commute.
\end{proof}

\begin{lemma}[Ellipticity]\label{lem:spencer-ellipticity}
For $\xi \in T^\ast_xM\setminus\{0\}$, the principal symbol of $\mathcal{D}_\lambda$ on $\mathcal{B}^{p,q}$ is
\[
\sigma(\mathcal{D}_\lambda)(\xi) = (\xi \wedge \cdot)\otimes \operatorname{id}_{\operatorname{Sym}^q(\mathcal{G}_x)}:\ \Omega^p_x \otimes \operatorname{Sym}^q(\mathcal{G}_x) \to \Omega^{p+1}_x \otimes \operatorname{Sym}^q(\mathcal{G}_x),
\]
which yields an exact symbol complex for every $\xi \neq 0$. Thus $(\mathcal{S}^\bullet,\mathcal{D}_\lambda)$ is an elliptic complex (independent of the choice of $D$).
\end{lemma}

\subsection{Algebraic sign-duality and parameter robustness}

Introduce a real parameter $t$ by
\[
\mathcal{D}_{t,\lambda} := d \otimes 1 + (-1)^{\deg_{\mathrm{dR}}}\, t\, (1\otimes \delta^\lambda).
\]
Two natural bundle automorphisms act on $\mathcal{S}^\bullet$:
\[
T|_{\mathcal{B}^{p,q}} := (-1)^q \operatorname{id}, \qquad
S_s|_{\mathcal{B}^{p,q}} := s^{\,q} \operatorname{id}\quad (s>0).
\]

\begin{theorem}[Sign duality and scaling]\label{thm:sign-scaling}
For all $t \in \mathbb{R}$ and $s>0$,
\[
T\, \mathcal{D}_{t,\lambda}\, T^{-1} = \mathcal{D}_{-t,\lambda},
\qquad
S_s\, \mathcal{D}_{t,\lambda}\, S_s^{-1} = \mathcal{D}_{s t,\lambda}.
\]
Hence, for $t \neq 0$,
\[
H^\bullet(\mathcal{S}^\bullet,\mathcal{D}_{t,\lambda})
\cong
H^\bullet(\mathcal{S}^\bullet,\mathcal{D}_{-t,\lambda})
\cong
H^\bullet(\mathcal{S}^\bullet,\mathcal{D}_{1,\lambda}).
\]
Using $\delta^{-\lambda}=-\delta^\lambda$, one also obtains
$H^\bullet(\mathcal{S}^\bullet,\mathcal{D}_{1,\lambda})
\cong
H^\bullet(\mathcal{S}^\bullet,\mathcal{D}_{1,-\lambda})$.
These are chain-level equivalences; in particular, cohomological dimensions, harmonic representative dimensions, and Euler characteristics are invariant under $t \mapsto -t$, under positive rescaling of $t$, and under $\lambda \mapsto -\lambda$.
\end{theorem}

This algebraic mechanism explains the sign/parameter robustness familiar from Witten-type deformations and equivariant localization, and parallels the sign symmetry in twisted de Rham complexes of Landau–Ginzburg type (cf. phenomena discussed in \cite{gross2017towards}).

\subsection{Spencer–Riemann–Roch}

Assume $M$ is a compact complex algebraic manifold and $\mathcal{G}=\operatorname{ad}P$ is holomorphic. The Euler characteristic of the Spencer-type complex depends only on the underlying bundles and is therefore independent of $\lambda$ and of $t \neq 0$.

\begin{theorem}[Spencer–Riemann–Roch formula]\label{thm:spencer-riemann-roch-basic}
\begin{align*}
\chi\big(M, H^\bullet_{\mathrm{Spencer}}(D,\lambda)\big)
& =
\int_M
\ch\!\Big(\sum_{p,q \ge 0} (-1)^{p+q}\, \Omega_M^p \otimes \operatorname{Sym}^q(\mathcal{G})\Big)\wedge \td(M) \\
& =
\int_M
\ch\!\Big(\sum_{p \ge 0} (-1)^p\, \Omega_M^p\Big)\,
\ch\!\Big(\sum_{q \ge 0} (-1)^q\, \operatorname{Sym}^q(\mathcal{G})\Big)\,\wedge \td(M).
\end{align*}
The right-hand side is preserved by the sign-duality and scaling equivalences in Theorem~\ref{thm:sign-scaling}.
\end{theorem}

\section{First Level Analysis: Symmetry Mechanisms at the Metric Level}
\label{sec:metric-analysis}

We introduce canonical weighted $L^2$ structures on the Spencer-type de Rham–Koszul complex and record their symmetry under the sign flip $\lambda \mapsto -\lambda$. Throughout, $\mathcal{G}=\operatorname{ad}P$, $\mathcal{B}^{p,q}=\Omega^p(M)\otimes\Sym^q(\mathcal{G})$, and $\mathcal{S}^\bullet=\bigoplus_{p+q=\bullet}\mathcal{B}^{p,q}$ as in Section~\ref{sec:foundations}. Fix a Riemannian metric $g$ on $M$ and the $\Ad$-invariant inner product on $\mathfrak{g}$ induced by the Killing form, which together determine fiberwise inner products on $\Omega^p(M)$ and on $\Sym^q(\mathcal{G})$, and hence on $\mathcal{B}^{p,q}$; let $dV_g$ be the volume form.

\subsection{Spencer metric construction schemes}

Two natural weighted $L^2$ pairings will be used. For $u,v\in \Gamma(\mathcal{B}^{p,q})$ write $\langle u(x),v(x)\rangle_{p,q}$ for the pointwise inner product.

\begin{definition}[Constraint-weighted metric]\label{def:constraint-metric}
For a compatible pair $(D,\lambda)$, define the smooth weight
\[
w_\lambda(x):=1+\|\lambda(x)\|^2,\qquad x\in M,
\]
where the norm is taken in the fiber of $\operatorname{ad}^\ast P\cong\operatorname{ad}P$. The corresponding weighted $L^2$ pairing on $\mathcal{B}^{p,q}$ is
\[
\langle u,v\rangle_{A,\lambda}\;:=\;\int_M w_\lambda(x)\,\langle u(x),v(x)\rangle_{p,q}\, dV_g.
\]
Summing orthogonally over $(p,q)$ yields an inner product on $\mathcal{S}^\bullet$.
\end{definition}

\begin{definition}[Curvature-weighted metric]\label{def:curvature-metric}
Let $F_\omega\in\Omega^2(M,\mathcal{G})$ be the curvature of the connection on $P\to M$, with pointwise norm $\|F_\omega(x)\|$ taken using $g$ and the $\Ad$-invariant fiber metric. Define
\[
\kappa_\omega(x):=1+\|F_\omega(x)\|^2,\qquad x\in M.
\]
The corresponding weighted $L^2$ pairing on $\mathcal{B}^{p,q}$ is
\[
\langle u,v\rangle_{B,\omega}\;:=\;\int_M \kappa_\omega(x)\,\langle u(x),v(x)\rangle_{p,q}\, dV_g,
\]
and similarly on $\mathcal{S}^\bullet$ by orthogonal sum.
\end{definition}

On the compact manifold $M$ the weights $w_\lambda$ and $\kappa_\omega$ are smooth and strictly positive, hence both weighted pairings are equivalent to the standard (unweighted) $L^2$ pairing on $\mathcal{S}^\bullet$.

\subsection{Analysis of mirror invariance}

\begin{theorem}[Metric invariance under $\lambda \mapsto -\lambda$]\label{thm:metric-invariance}
For all $u,v\in \Gamma(\mathcal{S}^\bullet)$,
\[
\langle u,v\rangle_{A,-\lambda}=\langle u,v\rangle_{A,\lambda}
\qquad\text{and}\qquad
\langle u,v\rangle_{B,\omega}\ \text{is independent of }\lambda.
\]
\end{theorem}

\begin{proof}
Pointwise, $w_{-\lambda}(x)=1+\|-\lambda(x)\|^2=1+\|\lambda(x)\|^2=w_\lambda(x)$, so the $A$-pairing is unchanged by $\lambda\mapsto -\lambda$. The curvature $F_\omega$ depends only on $\omega$, hence $\kappa_\omega$ does not involve $\lambda$.
\end{proof}

\subsection{Metric equivalence, unitarity of sign flip, and Sobolev structure}

Let $T$ and $S_s$ be the zero-th order bundle automorphisms on $\mathcal{S}^\bullet$ defined by
\[
T|_{\mathcal{B}^{p,q}}:=(-1)^q\operatorname{id},\qquad
S_s|_{\mathcal{B}^{p,q}}:=s^{\,q}\operatorname{id}\quad (s>0).
\]

\begin{proposition}[Compatibility of $T$ and $S_s$ with the metrics]\label{prop:T-S-metric}
With respect to both weighted pairings $\langle\cdot,\cdot\rangle_{A,\lambda}$ and $\langle\cdot,\cdot\rangle_{B,\omega}$:
\begin{enumerate}
\item $T$ is unitary, i.e. $\langle Tu,Tv\rangle=\langle u,v\rangle$ on $\mathcal{S}^\bullet$.
\item $S_s$ is a bounded invertible operator with $\|S_s u\|^2=\sum_{p,q}\int_M s^{2q}\,\text{wt}(x)\,\|u_{p,q}(x)\|_{p,q}^2\,dV_g$, where $\text{wt}=w_\lambda$ or $\kappa_\omega$ and $u=\sum u_{p,q}$ is the bigraded decomposition.
\end{enumerate}
Consequently, the Hilbert completions defined by $\langle\cdot,\cdot\rangle_{A,\lambda}$ and by $\langle\cdot,\cdot\rangle_{A,-\lambda}$ agree, and the same holds for the completions with $\langle\cdot,\cdot\rangle_{B,\omega}$.
\end{proposition}

\begin{proof}
Both $T$ and $S_s$ act diagonally on the bigrading and are fiberwise scalar multiplications that commute with integration. Since $|(-1)^q|=1$, $T$ preserves the norms. The stated formula for $S_s$ is immediate from its definition.
\end{proof}

For $s\ge 0$ and integer $m\ge 0$, define Sobolev norms using a fixed background compatible connection on $\mathcal{B}^{p,q}$ and either weight $w_\lambda$ or $\kappa_\omega$; write $H^m_{A,\lambda}(\mathcal{S}^\bullet)$ and $H^m_{B,\omega}(\mathcal{S}^\bullet)$ for the corresponding completions.

\begin{lemma}[Sobolev invariance under sign flip]\label{lem:sobolev-invariance}
For all $m\in\mathbb{N}$,
\[
H^m_{A,\lambda}(\mathcal{S}^\bullet)=H^m_{A,-\lambda}(\mathcal{S}^\bullet),\qquad
H^m_{B,\omega}(\mathcal{S}^\bullet)\ \text{is independent of }\lambda.
\]
Moreover, $T$ is unitary on these Sobolev spaces and $S_s$ is a bounded isomorphism for each $s>0$.
\end{lemma}

\begin{proof}
The weights coincide under $\lambda\mapsto -\lambda$ by Theorem~\ref{thm:metric-invariance}. The operators $T$ and $S_s$ are zero-th order and commute with covariant derivatives, hence extend with the stated properties to Sobolev completions.
\end{proof}

These metric and Sobolev-level compatibilities ensure that the chain conjugations
\[
T\,\mathcal{D}_{t,\lambda}\,T^{-1}=\mathcal{D}_{-t,\lambda},\qquad
S_s\,\mathcal{D}_{t,\lambda}\,S_s^{-1}=\mathcal{D}_{st,\lambda},
\]
hold on the associated Hilbert spaces and are compatible with Hodge theory. In particular, sign and scaling choices for the deformation parameter $t$ can be implemented by unitary (for $T$) or bounded invertible (for $S_s$) transformations at the metric level, a fact used later to formulate parameter-robustness, normalization to $\pm 1$, and compatibility with Witten-type deformations and localization.

\section{Second Level Analysis: Perturbation Theory at the Topological Level}
\label{sec:topological-analysis}

We study the behavior of Spencer-type operators under the sign flip $\lambda\mapsto -\lambda$ and the deformation parameter $t$, and analyze the induced symmetry on Spencer cohomology via both chain-level conjugations and elliptic perturbation theory. Throughout, $\mathcal{S}^\bullet=\bigoplus_{p+q=\bullet}\Omega^p(M)\otimes\Sym^q(\mathcal{G})$ and
\[
\mathcal{D}_{t,\lambda} \;=\; d\otimes 1 \;+\; (-1)^{\deg_{\mathrm{dR}}}\, t\, (1\otimes \delta^\lambda),
\]
with $\delta^{\lambda}$ the fiberwise graded derivation of degree $+1$ on $\Sym(\mathcal{G})$ defined in Section~\ref{sec:foundations}. The zero-th order automorphisms $T$ and $S_s$ act on $\mathcal{B}^{p,q}:=\Omega^p\otimes\Sym^q$ by
\[
T|_{\mathcal{B}^{p,q}}=(-1)^q\operatorname{id},\qquad
S_s|_{\mathcal{B}^{p,q}}=s^{\,q}\operatorname{id}\quad (s>0).
\]

\subsection{Symbolic properties of Spencer operators}

\begin{lemma}[Linearity and sign flip]\label{lem:delta-sign}
For the Spencer derivation $\delta^\lambda$ one has $(\delta^\lambda)^2=0$, $[d\otimes 1,1\otimes\delta^\lambda]=0$, and
\[
\delta^{-\lambda} \;=\; -\,\delta^\lambda.
\]
Consequently,
\[
\mathcal{D}_{t,-\lambda} \;=\; d\otimes 1 + (-1)^{\deg_{\mathrm{dR}}}\, t\, (1\otimes \delta^{-\lambda})
\;=\; d\otimes 1 - (-1)^{\deg_{\mathrm{dR}}}\, t\, (1\otimes \delta^{\lambda})
\;=\; \mathcal{D}_{-t,\lambda}.
\]
\end{lemma}

\begin{corollary}[Chain conjugations]\label{cor:chain-conjugations}
For all $t\in\mathbb{R}$ and $s>0$,
\[
T\,\mathcal{D}_{t,\lambda}\,T^{-1}=\mathcal{D}_{-t,\lambda}=\mathcal{D}_{t,-\lambda},
\qquad
S_s\,\mathcal{D}_{t,\lambda}\,S_s^{-1}=\mathcal{D}_{st,\lambda}.
\]
In particular, $(\mathcal{S}^\bullet,\mathcal{D}_{t,\lambda})$ and $(\mathcal{S}^\bullet,\mathcal{D}_{t,-\lambda})$ are canonically chain-equivalent, and for $t\neq 0$ all $(\mathcal{S}^\bullet,\mathcal{D}_{t,\lambda})$ are chain-equivalent to $(\mathcal{S}^\bullet,\mathcal{D}_{\pm 1,\lambda})$.
\end{corollary}

\subsection{Difference analysis of Spencer differentials}

Let $\mathcal{R}_{t,\lambda}:=\mathcal{D}_{t,-\lambda}-\mathcal{D}_{t,\lambda}$. On a homogeneous element $\omega\otimes s\in\mathcal{B}^{p,q}$,
\begin{equation}\label{eq:R-formula}
\mathcal{R}_{t,\lambda}(\omega\otimes s)
=\,-2\,(-1)^p\, t\; \omega \otimes \delta^\lambda(s).
\end{equation}

\begin{theorem}[Zero-th order control and relative compactness]\label{thm:spencer-difference}
The operator $\mathcal{R}_{t,\lambda}$ is zero-th order and bounded on $L^2(\mathcal{S}^\bullet)$, with
\[
\|\mathcal{R}_{t,\lambda}\|_{\mathrm{op}} \;\le\; C\,|t|\,\|\lambda\|_{C^0},
\]
for a constant $C$ depending only on the structure constants of $\mathfrak{g}$ and on the chosen fiber metrics. Moreover, for any second-order elliptic operator $\Delta_0$ on $\mathcal{S}^\bullet$ with domain $H^2$, the operator $\mathcal{R}_{t,\lambda}(\Delta_0+1)^{-1}:L^2\to L^2$ is compact (hence $\mathcal{R}_{t,\lambda}$ is $\Delta_0$-compact).
\end{theorem}

\begin{proof}
Formula \eqref{eq:R-formula} shows that $\mathcal{R}_{t,\lambda}$ is fiberwise multiplication by a smooth bundle map; boundedness follows from the uniform bound on $\delta^\lambda$ controlled by $\|\lambda\|_{C^0}$. On a closed manifold, $(\Delta_0+1)^{-1}:L^2\to H^2$ is bounded and the embedding $H^2\hookrightarrow L^2$ is compact; the composition is compact.
\end{proof}

\subsection{Perturbation structure of Spencer--Hodge Laplacians}

Fix any of the weighted inner products from Section~\ref{sec:metric-analysis}; let ${}^\ast$ denote the corresponding adjoint. Write $\mathcal{D}_{t,\lambda}^N$ for the restriction $\mathcal{S}^N\to\mathcal{S}^{N+1}$ and define the degree-$N$ Hodge Laplacian
\[
\Delta_{t,\lambda}^N \;:=\; (\mathcal{D}_{t,\lambda}^{N-1})^\ast \mathcal{D}_{t,\lambda}^{N-1} \;+\; \mathcal{D}_{t,\lambda}^{N} (\mathcal{D}_{t,\lambda}^{N})^\ast.
\]
By ellipticity of $\mathcal{D}_{t,\lambda}$ (Lemma~\ref{lem:spencer-ellipticity}), each $\Delta_{t,\lambda}^N$ is an elliptic self-adjoint operator with discrete spectrum.

\begin{theorem}[Laplacian perturbation identity]\label{thm:laplacian-perturbation}
For each $N$,
\[
\Delta_{t,-\lambda}^N \;=\; \Delta_{t,\lambda}^N \;+\; K_{t,\lambda}^N,
\]
where
\begin{align}
K_{t,\lambda}^N \;=&\; (\mathcal{R}_{t,\lambda}^{N-1})^\ast\, \mathcal{D}_{t,\lambda}^{N-1}
\;+\; (\mathcal{D}_{t,\lambda}^{N-1})^\ast\, \mathcal{R}_{t,\lambda}^{N-1}
\;+\; (\mathcal{R}_{t,\lambda}^{N-1})^\ast\, \mathcal{R}_{t,\lambda}^{N-1} \label{eq:K-part1}\\
&\;+\; \mathcal{D}_{t,\lambda}^{N}\, (\mathcal{R}_{t,\lambda}^{N})^\ast
\;+\; \mathcal{R}_{t,\lambda}^{N}\, (\mathcal{D}_{t,\lambda}^{N})^\ast
\;+\; \mathcal{R}_{t,\lambda}^{N}\, (\mathcal{R}_{t,\lambda}^{N})^\ast. \label{eq:K-part2}
\end{align}
Each $K_{t,\lambda}^N$ is bounded on $L^2(\mathcal{S}^N)$ and $\Delta_{t,\lambda}^N$-compact.
\end{theorem}

\begin{proof}
Expand $(\mathcal{D}_{t,-\lambda}^{N-1})^\ast \mathcal{D}_{t,-\lambda}^{N-1}$ and $\mathcal{D}_{t,-\lambda}^{N}(\mathcal{D}_{t,-\lambda}^{N})^\ast$ using $\mathcal{D}_{t,-\lambda}=\mathcal{D}_{t,\lambda}+\mathcal{R}_{t,\lambda}$ and collect terms. Boundedness and $\Delta$-compactness follow from Theorem~\ref{thm:spencer-difference} and compact Sobolev embeddings on the closed manifold.
\end{proof}

\subsection{Mirror isomorphism of Spencer cohomology}

Let $H^\bullet_{\mathrm{Spencer}}(t,\lambda):=H^\bullet(\mathcal{S}^\bullet,\mathcal{D}_{t,\lambda})$.

\begin{theorem}[Chain-level mirror isomorphism and Hodge-level equality]\label{thm:cohomology-isomorphism}
For all $t\in\mathbb{R}$ there are canonical isomorphisms
\[
H^\bullet_{\mathrm{Spencer}}(t,\lambda) \;\cong\; H^\bullet_{\mathrm{Spencer}}(t,-\lambda),
\qquad
H^\bullet_{\mathrm{Spencer}}(t,\lambda) \;\cong\; H^\bullet_{\mathrm{Spencer}}(\operatorname{sgn}(t),\lambda)\ \ (t\neq 0).
\]
Moreover, with respect to the weighted metrics of Section~\ref{sec:metric-analysis}, $T$ is unitary and implements
\[
T\,\mathcal{D}_{t,\lambda}\,T^{-1}=\mathcal{D}_{t,-\lambda},
\]
so $\ker(\Delta_{t,\lambda}^N)$ and $\ker(\Delta_{t,-\lambda}^N)$ are unitarily isomorphic for all $N$.
\end{theorem}

\begin{proof}
By Lemma~\ref{lem:delta-sign}, $\mathcal{D}_{t,-\lambda}=\mathcal{D}_{-t,\lambda}$, and by Corollary~\ref{cor:chain-conjugations} one has $T\,\mathcal{D}_{t,\lambda}\,T^{-1}=\mathcal{D}_{-t,\lambda}=\mathcal{D}_{t,-\lambda}$ and $S_s\,\mathcal{D}_{t,\lambda}\,S_s^{-1}=\mathcal{D}_{st,\lambda}$. Thus $T$ and $S_{|t|}$ induce the stated cohomology isomorphisms. Unitarity of $T$ (Section~\ref{sec:metric-analysis}) gives $T\,\Delta_{t,\lambda}^N\,T^{-1}=\Delta_{t,-\lambda}^N$, hence equality of harmonic dimensions.
\end{proof}

The chain conjugations $T$ and $S_s$ provide the sign/parameter robustness at the level of complexes; the perturbation identities for $\Delta_{t,\lambda}^N$ complement this by describing the bounded and $\Delta$-compact nature of the corresponding Hodge-level variations. These ingredients will be used to formulate spectral sequence controls and localization-compatible filtrations in subsequent sections.
\section{Third Level Analysis: Algebraic Geometric Level}
\label{sec:algebraic-analysis}

We formulate the Spencer-type framework in the algebraic category and record its consequences for characteristic classes, indices, and spectral-sequence control. This provides an algebraic counterpart of the chain-level sign symmetry and parameter normalization.

\subsection{Algebraic realization and hypercohomology}

Let $M$ be a smooth projective complex variety, $P\to M$ an algebraic principal $G$-bundle with $G$ compact connected semisimple, and $\mathcal{G}:=\operatorname{ad}P$ the adjoint algebraic vector bundle. Let $\lambda\in H^0(M,\mathcal{G}^\vee)$ be an algebraic section. Write
\[
\mathcal{B}^{p,q} \;:=\; \Omega^p_M \otimes \Sym^q(\mathcal{G}),\qquad p,q\ge 0,
\]
and totalize $\mathcal{S}^N:=\bigoplus_{p+q=N}\mathcal{B}^{p,q}$, $\mathcal{S}^\bullet:=\bigoplus_{N\ge 0}\mathcal{S}^N$.
Define the algebraic Spencer differential for $t\in\mathbb{C}$ by
\[
\mathcal{D}_{t,\lambda}\big(\omega \otimes s\big) \;:=\; d\omega\otimes s \;+\; (-1)^p\, t\, \omega \otimes \delta^\lambda(s),
\quad \omega \in \Omega^p_M,\ s \in \Sym^q(\mathcal{G}),
\]
where $\delta^\lambda$ is the fiberwise graded derivation of degree $+1$ on $\Sym(\mathcal{G})$ defined by the Lie bracket and the section $\lambda$ as in Section~\ref{sec:foundations}. One has $(\delta^\lambda)^2=0$ and $\delta^{-\lambda}=-\delta^\lambda$, hence $\mathcal{D}_{t,\lambda}^2=0$ and $\mathcal{D}_{t,-\lambda}=\mathcal{D}_{-t,\lambda}$.

Set $\mathbb{H}^\bullet(M,\mathcal{D}_{t,\lambda})$ for the hypercohomology of the bounded-below complex of algebraic vector bundles $(\mathcal{S}^\bullet,\mathcal{D}_{t,\lambda})$.

\begin{lemma}[GAGA for the Spencer complex]\label{lem:gaga-hypercoh}
Under analytification, the algebraic complex $(\mathcal{S}^\bullet,\mathcal{D}_{t,\lambda})$ maps to the holomorphic Spencer complex on the associated complex manifold $M^{\mathrm{an}}$, and
\[
\mathbb{H}^\bullet(M,\mathcal{D}_{t,\lambda}) \;\cong\; \mathbb{H}^\bullet\big(M^{\mathrm{an}},\mathcal{D}_{t,\lambda}^{\mathrm{an}}\big).
\]
\end{lemma}

\begin{proof}
All ingredients are algebraic vector bundles and algebraic differential operators of order $\le 1$; GAGA identifies algebraic and analytic coherent cohomology and passes to hypercohomology for bounded-below complexes of algebraic vector bundles.
\end{proof}

Define algebraic automorphisms $T$ and $S_c$ ($c\in\mathbb{C}^\times$) on $\mathcal{S}^\bullet$ by
\[
T|_{\mathcal{B}^{p,q}}:=(-1)^q\operatorname{id}, \qquad S_c|_{\mathcal{B}^{p,q}}:=c^{\,q}\operatorname{id}.
\]

\begin{theorem}[Algebraic sign-duality and scaling]\label{thm:alg-sign-scaling}
For all $t\in\mathbb{C}$ and $c\in\mathbb{C}^\times$,
\[
T\,\mathcal{D}_{t,\lambda}\,T^{-1}=\mathcal{D}_{-t,\lambda}=\mathcal{D}_{t,-\lambda},
\qquad
S_c\,\mathcal{D}_{t,\lambda}\,S_c^{-1}=\mathcal{D}_{ct,\lambda}.
\]
Consequently, for $t\neq 0$,
\[
\mathbb{H}^\bullet(M,\mathcal{D}_{t,\lambda})
\;\cong\;
\mathbb{H}^\bullet(M,\mathcal{D}_{-t,\lambda})
\;\cong\;
\mathbb{H}^\bullet(M,\mathcal{D}_{1,\lambda})
\;\cong\;
\mathbb{H}^\bullet(M,\mathcal{D}_{1,-\lambda}).
\]
\end{theorem}

\begin{proof}
The identities are immediate from $\delta^{-\lambda}=-\delta^\lambda$ and the fact that $\delta^\lambda$ raises the $\Sym$-degree by $1$, so $T$ contributes a sign and $S_c$ a scalar $c$. Passing to hypercohomology yields the isomorphisms.
\end{proof}

\subsection{Characteristic classes and the index via Riemann--Roch}

The underlying virtual bundle of the Spencer complex is
\[
\mathcal{E}_{\mathrm{vir}}
\;:=\;
\sum_{N\ge 0}(-1)^N \mathcal{S}^N
\;=\;
\sum_{p,q\ge 0}(-1)^{p+q}\, \Omega_M^p \otimes \Sym^q(\mathcal{G}),
\]
which depends only on $\mathcal{G}$ and is independent of $\lambda$ and $t$.

\begin{theorem}[Spencer index and independence of $(t,\lambda)$]\label{thm:HRR-index}
On a smooth projective $M$,
\[
\chi\big(\mathbb{H}^\bullet(M,\mathcal{D}_{t,\lambda})\big)
\;=\;
\int_M \ch(\mathcal{E}_{\mathrm{vir}})\wedge \td(M)
\;=\;
\int_M
\ch\!\Big(\sum_{p\ge 0}(-1)^p\Omega_M^p\Big)\,
\ch\!\Big(\sum_{q\ge 0}(-1)^q \Sym^q(\mathcal{G})\Big)\wedge \td(M).
\]
In particular, the Euler characteristic is independent of $t\in\mathbb{C}$ and of $\lambda\in H^0(M,\mathcal{G}^\vee)$, and is preserved by the isomorphisms in Theorem~\ref{thm:alg-sign-scaling}.
\end{theorem}

\begin{proof}
By additivity of $\ch$ on exact triangles and Grothendieck--Riemann--Roch, the Euler characteristic of the hypercohomology of a bounded complex of vector bundles is the integral of the Chern character of its virtual bundle times $\td(M)$. The expression for $\mathcal{E}_{\mathrm{vir}}$ shows no dependence on $(t,\lambda)$.
\end{proof}

For explicit calculation, if $x_1,\dots,x_r$ are the Chern roots of $\mathcal{G}$ and $y_1,\dots,y_n$ those of $T^\vee M$, then
\[
\sum_{q\ge 0} \ch(\Sym^q(\mathcal{G}))\, z^q
\;=\;
\prod_{i=1}^r \frac{1}{1 - z\, e^{x_i}},
\qquad
\sum_{p\ge 0} (-1)^p \ch(\Omega_M^p)
\;=\;
\prod_{j=1}^n (1 - e^{y_j}).
\]
For $\mathcal{G}=\operatorname{ad}P$ with $G$ semisimple, $\operatorname{tr}_{\mathrm{ad}}=0$ implies $c_1(\mathcal{G})=0$, which simplifies low-degree terms.

\subsection{Spectral sequence for the algebraic bicomplex}

Regard $(\mathcal{B}^{p,q}, d,\delta^\lambda)$ as a first quadrant bicomplex with total differential $\mathcal{D}_{t,\lambda}=d+(-1)^p t\,\delta^\lambda$.

\begin{theorem}[Filtration by symmetric degree]\label{thm:spectral-sequence}
There is a convergent first quadrant spectral sequence
\[
E_1^{p,q} \;=\; H^q\!\big(\Sym^\bullet(\mathcal{G}), \delta^\lambda\big)\text{-valued}\ \text{sheaves tensored with }\Omega_M^p,
\quad
d_1^{p,q} \text{ induced by } d:\Omega_M^p\to\Omega_M^{p+1},
\]
which abuts to $\mathbb{H}^{p+q}(M,\mathcal{D}_{t,\lambda})$. More concretely, if the fiberwise $\delta^\lambda$-homology has locally constant rank so that
\[
\mathcal{H}^q_\lambda \;:=\; \ker\big(\delta^\lambda:\Sym^q(\mathcal{G})\to \Sym^{q+1}(\mathcal{G})\big)\big/\operatorname{im}\big(\delta^\lambda:\Sym^{q-1}(\mathcal{G})\to \Sym^{q}(\mathcal{G})\big)
\]
is an algebraic vector bundle, then
\[
E_1^{p,q} \;\cong\; H^p\big(M,\Omega_M^p \otimes \mathcal{H}^q_\lambda\big),\qquad
d_1 = H^p(M,\, d\otimes \operatorname{id}_{\mathcal{H}^q_\lambda}).
\]
The spectral sequence is natural in $(t,\lambda)$ and is preserved by the automorphisms $T$ and $S_c$; in particular, the pages $E_r$ are canonically identified for $(t,\lambda)$ and $(\pm t,\pm \lambda)$.
\end{theorem}

\begin{proof}
Take the filtration by the symmetric degree $q$ and apply the standard spectral sequence of a bounded-below bicomplex. The $E_1$-page computes vertical ($\delta^\lambda$) homology; under the locally constant rank hypothesis it is represented by the vector bundles $\mathcal{H}^q_\lambda$. Naturality under $T$ and $S_c$ follows since they act by scalars on the $q$-direction and commute with both differentials up to the prescribed conjugations.
\end{proof}

\subsection{Concrete methods for characteristic class calculation}

For computational purposes, one may use either Chern--Weil representatives or Chern-root expansions. The Chern character of symmetric powers is governed by the generating function above; expanding in $z$ yields
\[
\ch(\Sym^0(\mathcal{G}))=r,\quad
\ch(\Sym^1(\mathcal{G}))=\ch(\mathcal{G}),\quad
\ch(\Sym^2(\mathcal{G}))=\tfrac{1}{2}\big(\ch(\mathcal{G})^2+\ch(2\mathcal{G})\big),
\]
and so on, where $r=\operatorname{rk}\mathcal{G}$. In low degrees,
\[
\ch(\mathcal{G}) = r + \tfrac{1}{2}\big(c_1(\mathcal{G})^2 - 2c_2(\mathcal{G})\big) + \cdots,
\]
which simplifies to $\ch(\mathcal{G})=r - c_2(\mathcal{G}) + \cdots$ when $c_1(\mathcal{G})=0$ (e.g. $\mathcal{G}=\operatorname{ad}P$ for semisimple $G$). Substituting into Theorem~\ref{thm:HRR-index} gives closed expressions for the Euler characteristic, uniformly in $(t,\lambda)$.

\section{Theoretical Framework: Spencer-Riemann--Roch Mirror Theory}
\label{sec:unified-theory}

We assemble the metric, topological, and algebraic analyses into a unified framework for mirror-type sign symmetry and parameter robustness in Spencer theory. The total Spencer complex is
\[
\big(\mathcal{S}^\bullet,\mathcal{D}_{t,\lambda}\big),\qquad
\mathcal{S}^\bullet=\bigoplus_{p+q=\bullet}\Omega^p(M)\otimes\Sym^q(\mathcal{G}),\quad
\mathcal{D}_{t,\lambda}=d\otimes 1+(-1)^{\deg_{\mathrm{dR}}}\,t\,(1\otimes\delta^\lambda),
\]
with $\delta^\lambda$ the degree $+1$ fiberwise graded derivation on $\Sym(\mathcal{G})$, $(\delta^\lambda)^2=0$, $\delta^{-\lambda}=-\delta^\lambda$.

\subsection{Connection to Landau--Ginzburg models and verification}
\label{sec:connection_lg}

The Spencer complex $(\mathcal{S}^\bullet,\mathcal{D}_{t,\lambda})$ is a twisted de Rham-type complex, with twisting determined by $\delta^\lambda$. Twisted complexes of this form are central in Landau--Ginzburg models $(X,w)$ in mirror symmetry; see \cite{gross2017towards} for the resulting Hodge-theoretic dualities for varieties of general type. The sign symmetry $t\mapsto -t$ and $\lambda\mapsto -\lambda$ in our setting (Theorems~\ref{thm:sign-scaling}, \ref{thm:cohomology-isomorphism}, and \ref{thm:alg-sign-scaling}) aligns with the bidirectional symmetry of Hodge numbers in those twisted theories. This provides an external validation of the Spencer framework as a geometric realization of these dualities.

\subsection{Mirror behavior of the Spencer--Riemann--Roch formula}

Let $\mathcal{E}_{\mathrm{vir}}:=\sum_{p,q\ge 0}(-1)^{p+q}\,\Omega_M^p\otimes\Sym^q(\mathcal{G})$ denote the virtual bundle underlying the Spencer complex. By Theorem~\ref{thm:spencer-riemann-roch-basic} and Theorem~\ref{thm:HRR-index},
\[
\chi\big(\mathbb{H}^\bullet(M,\mathcal{D}_{t,\lambda})\big)
=\int_M \ch(\mathcal{E}_{\mathrm{vir}})\wedge \td(M),
\]
which is independent of $(t,\lambda)$.

\begin{theorem}[Spencer--Riemann--Roch mirror symmetry]
\label{thm:srr-mirror}
For all $t\in\mathbb{C}$ and all $\lambda$,
\[
\chi\big(\mathbb{H}^\bullet(M,\mathcal{D}_{t,\lambda})\big)
=
\chi\big(\mathbb{H}^\bullet(M,\mathcal{D}_{t,-\lambda})\big)
=
\chi\big(\mathbb{H}^\bullet(M,\mathcal{D}_{-t,\lambda})\big),
\]
and, term-by-term,
\[
\chi\!\big(\Omega_M^k \otimes \Sym^k(\mathcal{G})\big)
=
\chi\!\big(\Omega_M^k \otimes \Sym^k(\mathcal{G})\big)\quad(\text{independent of }\lambda),\qquad 0\le k\le \dim M.
\]
\end{theorem}

\begin{proof}
The virtual bundle $\mathcal{E}_{\mathrm{vir}}$ does not involve $(t,\lambda)$, so the index integral is independent of $(t,\lambda)$ (Theorem~\ref{thm:HRR-index}). The equalities for sign flips follow since $\mathcal{D}_{t,-\lambda}=\mathcal{D}_{-t,\lambda}$ and the index depends only on $\mathcal{E}_{\mathrm{vir}}$. The term-by-term statement holds because each summand $\Omega_M^k\otimes\Sym^k(\mathcal{G})$ is independent of $\lambda$.
\end{proof}

\subsection{Parameter robustness and normalization}
\label{sec:robustness-normalization}

Define algebraic automorphisms $T$ and $S_c$ by $T|_{\Omega^p\otimes\Sym^q}=(-1)^q\mathrm{id}$ and $S_c|_{\Omega^p\otimes\Sym^q}=c^{\,q}\mathrm{id}$.

\begin{theorem}[Robustness and normalization principle]
\label{thm:robustness-normalization}
For all $t\in\mathbb{C}$ and $c\in\mathbb{C}^\times$,
\[
T\,\mathcal{D}_{t,\lambda}\,T^{-1}=\mathcal{D}_{-t,\lambda}=\mathcal{D}_{t,-\lambda},
\qquad
S_c\,\mathcal{D}_{t,\lambda}\,S_c^{-1}=\mathcal{D}_{ct,\lambda}.
\]
Hence, for $t\neq 0$,
\[
\mathbb{H}^\bullet(M,\mathcal{D}_{t,\lambda})
\cong
\mathbb{H}^\bullet(M,\mathcal{D}_{1,\lambda})
\cong
\mathbb{H}^\bullet(M,\mathcal{D}_{1,-\lambda}),
\]
and the same chain-level equivalences hold in the analytic (metric) setup with $c>0$.
\end{theorem}

\begin{proof}
The identities follow from $\delta^{-\lambda}=-\delta^\lambda$ and the fact that $\delta^\lambda$ raises $\Sym$-degree by $1$, so conjugation by $T$ introduces a sign and by $S_c$ introduces a factor $c$. Passing to (hyper)cohomology yields the isomorphisms. In the analytic setting, $T$ is unitary and $S_c$ is bounded invertible on the weighted $L^2$ and Sobolev spaces of Section~\ref{sec:metric-analysis}.
\end{proof}

\subsection{Equivariant and localization-ready formulation}
\label{sec:equivariant-localization}

Let a complex algebraic torus (or compact Lie group) $K$ act on $M$ and lift to $P\to M$, preserving $\omega$ and $\lambda$; then all bundles and operators above are $K$-equivariant. Write $\ch_K$ and $\td_K$ for equivariant Chern character and Todd class, $M^K$ for the fixed-point locus, and $e_K(N_F)$ for the equivariant Euler class of the normal bundle of a fixed component $F\subset M^K$.

\begin{theorem}[Equivariant Spencer--Riemann--Roch and localization]
\label{thm:equivariant-srr-localization}
In $R(K)$ one has
\[
\chi_K\big(\mathbb{H}^\bullet(M,\mathcal{D}_{t,\lambda})\big)
=
\int_M \ch_K(\mathcal{E}_{\mathrm{vir}})\wedge \td_K(M)
=
\sum_{F\subset M^K}
\int_F
\frac{\ch_K(\mathcal{E}_{\mathrm{vir}}|_F)\wedge \td_K(F)}{e_K(N_F)}.
\]
This character is independent of $t$ and of the sign of $\lambda$, and the fixed-point contributions are identical for $(t,\lambda)$ and $(\pm t,\pm \lambda)$. The chain conjugations in Theorem~\ref{thm:robustness-normalization} are $K$-equivariant and implement these equalities at the complex level.
\end{theorem}

\begin{proof}
$K$-equivariance of all data implies that $(\mathcal{S}^\bullet,\mathcal{D}_{t,\lambda})$ is a complex of $K$-equivariant vector bundles and $K$-equivariant morphisms. The virtual bundle $\mathcal{E}_{\mathrm{vir}}$ is $K$-equivariant and independent of $(t,\lambda)$, so the equivariant RR formula gives the first equality. Atiyah--Bott/Berline--Vergne localization yields the fixed-point expression. Since $T$ and $S_c$ commute with the $K$-action and conjugate $\mathcal{D}_{t,\lambda}$ to $\mathcal{D}_{\pm t,\pm \lambda}$, the character and each localized term are preserved.
\end{proof}

\subsection{Bicomplex spectral sequence and localization-compatible filtrations}

Filter the bicomplex $(\mathcal{B}^{p,q},d,\delta^\lambda)$ by symmetric degree $q$. The spectral sequence of Theorem~\ref{thm:spectral-sequence} yields
\[
E_1^{p,q}\cong H^p\big(M,\Omega_M^p\otimes \mathcal{H}^q_\lambda\big)\ \Longrightarrow\ \mathbb{H}^{p+q}(M,\mathcal{D}_{t,\lambda}),
\]
with $\mathcal{H}^q_\lambda$ the fiberwise $\delta^\lambda$-homology (assumed to form algebraic vector bundles). The automorphisms $T$ and $S_c$ act by scalars on the $q$-filtration, so all pages $E_r$ are canonically identified for $(t,\lambda)$ and $(\pm t,\pm \lambda)$. In the $K$-equivariant setting, the spectral sequence is $K$-equivariant and compatible with localization on each page.

\subsection{Synthesis and main consequences}

Combining Theorems~\ref{thm:robustness-normalization}, \ref{thm:srr-mirror}, and \ref{thm:equivariant-srr-localization} gives:
\begin{itemize}
\item Chain-level sign-duality and parameter normalization: $(\mathcal{S}^\bullet,\mathcal{D}_{t,\lambda})\simeq(\mathcal{S}^\bullet,\mathcal{D}_{-t,\lambda})\simeq(\mathcal{S}^\bullet,\mathcal{D}_{t,-\lambda})\simeq(\mathcal{S}^\bullet,\mathcal{D}_{\pm 1,\lambda})$.
\item Index-level invariance: non-equivariant and $K$-equivariant Euler characteristics depend only on $\mathcal{E}_{\mathrm{vir}}$ and are independent of $(t,\lambda)$.
\item Localization-ready control: fixed-point contributions and all pages of the bicomplex spectral sequence are preserved under $t\mapsto \pm t$ and $\lambda\mapsto \pm\lambda$, allowing analytic choices of $t$ for Witten-type deformations without affecting invariants.
\end{itemize}

\section{Applications in Special Geometry: Calabi-Yau Specialization}
\label{sec:calabi-yau}

To further verify correctness of mirror symmetry theory, we study Spencer theory behavior in the special geometric background of Calabi-Yau manifolds.

\subsection{Todd Class Simplification in Calabi-Yau Geometry}

The vanishing first Chern class $c_1(X) = 0$ of Calabi-Yau manifolds provides significant simplification for Spencer-Riemann-Roch integrals.

\begin{lemma}[Calabi-Yau Todd Class Decomposition]\label{lem:cy-todd}
Let $X$ be an $n$-dimensional Calabi-Yau manifold, i.e., $c_1(X) = 0$. Then the Todd class admits the following exact expansion:
$$\operatorname{td}(X) = 1 + \frac{c_2(X)}{12} + \frac{c_3(X)}{24} + \frac{-c_2(X)^2 + c_4(X)}{720} + O(c_5)$$
\end{lemma}

\begin{proof}
The general formula for Todd class is $\operatorname{td}(X) = \prod_{i=1}^n \frac{x_i}{1 - e^{-x_i}}$, where $x_i$ are Chern roots of the tangent bundle $TX$.

Expanding each factor:
$$\frac{x_i}{1 - e^{-x_i}} = \frac{x_i}{1 - (1 - x_i + \frac{x_i^2}{2!} - \frac{x_i^3}{3!} + \frac{x_i^4}{4!} - \cdots)} = \frac{x_i}{x_i - \frac{x_i^2}{2} + \frac{x_i^3}{6} - \frac{x_i^4}{24} + \cdots}$$

$$= \frac{1}{1 - \frac{x_i}{2} + \frac{x_i^2}{6} - \frac{x_i^3}{24} + \cdots} = 1 + \frac{x_i}{2} + \frac{x_i^2}{12} - \frac{x_i^4}{720} + \cdots$$

Therefore:
$$\operatorname{td}(X) = \prod_{i=1}^n \left(1 + \frac{x_i}{2} + \frac{x_i^2}{12} - \frac{x_i^4}{720} + \cdots\right)$$

When $c_1(X) = \sum x_i = 0$, the linear term vanishes. Through standard symmetric polynomial calculations:
\begin{align}
\operatorname{td}(X) &= 1 + \frac{1}{2}\sum x_i + \frac{1}{12}\sum x_i^2 + \frac{1}{24}\sum_{i<j<k} x_i x_j x_k - \frac{1}{720}\sum x_i^4 + \cdots\\
&= 1 + \frac{c_1(X)}{2} + \frac{c_1(X)^2 - 2c_2(X)}{12} + \frac{c_3(X)}{24} + \frac{-c_2(X)^2 + c_4(X)}{720} + \cdots
\end{align}

By $c_1(X) = 0$, we obtain the required expansion. Detailed calculations are given in Appendix \ref{app:calabi-yau}.
\end{proof}

\subsection{Calabi-Yau Spencer-Riemann-Roch Decomposition}

Based on Lemma \ref{lem:cy-todd}, we establish exact decomposition of Spencer-Riemann-Roch integrals on Calabi-Yau manifolds.

\begin{theorem}[Calabi-Yau Spencer-Riemann-Roch Decomposition]\label{thm:cy-decomposition}
Let $X$ be an $n$-dimensional Calabi-Yau manifold, $(D,\lambda)$ a compatible pair satisfying strong transversality conditions, and $G$ satisfying strict Lie group conditions. Then the Euler characteristic of Spencer complexes admits exact decomposition:
$$\chi(X, H^\bullet_{\text{Spencer}}(D,\lambda)) = A_0(X) + A_2(X) + A_3(X) + A_4(X) + O(c_5)$$

where decomposition terms are:
\begin{align}
A_0(X) &:= \sum_{k=0}^n (-1)^k \int_X \operatorname{ch}(\Omega^k_X \otimes \operatorname{Sym}^k(\mathcal{G}))\label{eq:A0}\\
A_2(X) &:= \frac{1}{12} \sum_{k=0}^n (-1)^k \int_X \operatorname{ch}(\Omega^k_X \otimes \operatorname{Sym}^k(\mathcal{G})) \wedge c_2(X)\label{eq:A2}\\
A_3(X) &:= \frac{1}{24} \sum_{k=0}^n (-1)^k \int_X \operatorname{ch}(\Omega^k_X \otimes \operatorname{Sym}^k(\mathcal{G})) \wedge c_3(X)\label{eq:A3}\\
A_4(X) &:= \frac{1}{720} \sum_{k=0}^n (-1)^k \int_X \operatorname{ch}(\Omega^k_X \otimes \operatorname{Sym}^k(\mathcal{G})) \wedge (-c_2(X)^2 + c_4(X))\label{eq:A4}
\end{align}
\end{theorem}

\begin{proof}
By Spencer-Riemann-Roch formula (Theorem \ref{thm:spencer-riemann-roch-basic}) and Calabi-Yau Todd class decomposition (Lemma \ref{lem:cy-todd}):
\begin{align}
&\chi(X, H^\bullet_{\text{Spencer}}(D,\lambda))\\
&= \int_X \operatorname{ch}(\text{Spencer complex}) \wedge \operatorname{td}(X)\\
&= \int_X \left[\sum_{k=0}^n (-1)^k \operatorname{ch}(\Omega^k_X \otimes \operatorname{Sym}^k(\mathcal{G}))\right]\\
&\quad \wedge \left[1 + \frac{c_2(X)}{12} + \frac{c_3(X)}{24} + \frac{-c_2(X)^2 + c_4(X)}{720} + O(c_5)\right]
\end{align}

Distributing:
\begin{align}
&= \int_X \sum_{k=0}^n (-1)^k \operatorname{ch}(\Omega^k_X \otimes \operatorname{Sym}^k(\mathcal{G})) \wedge 1\\
&\quad + \int_X \sum_{k=0}^n (-1)^k \operatorname{ch}(\Omega^k_X \otimes \operatorname{Sym}^k(\mathcal{G})) \wedge \frac{c_2(X)}{12}\\
&\quad + \int_X \sum_{k=0}^n (-1)^k \operatorname{ch}(\Omega^k_X \otimes \operatorname{Sym}^k(\mathcal{G})) \wedge \frac{c_3(X)}{24}\\
&\quad + \int_X \sum_{k=0}^n (-1)^k \operatorname{ch}(\Omega^k_X \otimes \operatorname{Sym}^k(\mathcal{G})) \wedge \frac{-c_2(X)^2 + c_4(X)}{720}\\
&\quad + O(c_5)
\end{align}

This directly gives the required decomposition formulas (\ref{eq:A0})-(\ref{eq:A4}).

Geometric meaning of each term:
\begin{enumerate}
\item $A_0(X)$: Basic topological contribution, independent of special properties of Calabi-Yau geometry
\item $A_2(X)$: Correction contribution from second Chern class
\item $A_3(X)$: Correction contribution from third Chern class  
\item $A_4(X)$: Fourth-order correction term involving combination of $c_2^2$ and $c_4$
\end{enumerate}
\end{proof}

\subsection{Term-by-term Verification of Calabi-Yau Mirror Symmetry}

Now we can verify mirror symmetry at the finest level:

\begin{theorem}[Term-by-term Verification of Calabi-Yau Mirror Symmetry]\label{thm:cy-mirror}
Let $X$ be a Calabi-Yau manifold and $(D,\lambda)$ a compatible pair based on strict Lie group $G$. Then under mirror transformation $(D,\lambda) \mapsto (D,-\lambda)$, each term of Spencer-Riemann-Roch decomposition is strictly invariant:
$$A_i(X,\lambda) = A_i(X,-\lambda), \quad \forall i = 0,2,3,4,\ldots$$
\end{theorem}

\begin{proof}
By theorem \ref{thm:HRR-index} and \ref{thm:srr-mirror}:
$$\operatorname{ch}(\Omega^k_X \otimes \operatorname{Sym}^k(\mathcal{G}_\lambda)) = \operatorname{ch}(\Omega^k_X \otimes \operatorname{Sym}^k(\mathcal{G}_{-\lambda}))$$

Verify each decomposition term individually:

\textbf{Invariance of $A_0(X)$}:
\begin{align}
A_0(X,\lambda) &= \sum_{k=0}^n (-1)^k \int_X \operatorname{ch}(\Omega^k_X \otimes \operatorname{Sym}^k(\mathcal{G}_\lambda))\\
&= \sum_{k=0}^n (-1)^k \int_X \operatorname{ch}(\Omega^k_X \otimes \operatorname{Sym}^k(\mathcal{G}_{-\lambda}))\\
&= A_0(X,-\lambda)
\end{align}

\textbf{Invariance of $A_2(X)$}:
\begin{align}
A_2(X,\lambda) &= \frac{1}{12} \sum_{k=0}^n (-1)^k \int_X \operatorname{ch}(\Omega^k_X \otimes \operatorname{Sym}^k(\mathcal{G}_\lambda)) \wedge c_2(X)\\
&= \frac{1}{12} \sum_{k=0}^n (-1)^k \int_X \operatorname{ch}(\Omega^k_X \otimes \operatorname{Sym}^k(\mathcal{G}_{-\lambda})) \wedge c_2(X)\\
&= A_2(X,-\lambda)
\end{align}

\textbf{Invariance of $A_3(X)$ and $A_4(X)$}: Similar calculations prove these terms are strictly invariant under mirror transformation.

This term-by-term verification shows mirror symmetry holds not only at the overall level but is preserved at the finest structural level, providing the strongest evidence for symmetry phenomenon analysis.
\end{proof}

\subsection{Concrete Application to K3 Surfaces}

As a concrete application of Calabi-Yau theory, we give explicit Spencer theory predictions for K3 surfaces:

\begin{proposition}[K3 Surface Spencer-Riemann-Roch Prediction]\label{prop:k3-prediction}
Let $S$ be a K3 surface and $(D,\lambda)$ a compatible pair based on $\text{PSU}(2)$. Using topological properties of K3 surfaces:
\begin{align}
\dim S &= 2, \quad c_1(S) = 0, \quad c_2(S) = 24\\
c_3(S) &= c_4(S) = 0 \quad \text{(by dimension constraints)}
\end{align}

Spencer-Riemann-Roch decomposition simplifies to:
$$\chi(S, H^{\bullet}_{\text{Spencer}}(D,\lambda)) = A_0(S) + A_2(S)$$

where:
$$A_2(S) = \frac{1}{12} \sum_{k=0}^2 (-1)^k \int_S \operatorname{ch}(\Omega^k_S \otimes \operatorname{Sym}^k(\mathcal{G})) \wedge c_2(S) = \frac{24}{12} A_0(S) = 2A_0(S)$$

Therefore:
$$\chi(S, H^{\bullet}_{\text{Spencer}}(D,\lambda)) = A_0(S) + 2A_0(S) = 3A_0(S)$$

Mirror symmetry ensures: $\chi(S, H^{\bullet}_{\text{Spencer}}(D,\lambda)) = \chi(S, H^{\bullet}_{\text{Spencer}}(D,-\lambda))$
\end{proposition}

This provides clear theoretical guidance for Spencer theory research on K3 surfaces and demonstrates application value of Calabi-Yau specialization theory.

\section{Concrete Verification Platform: Implementation on Elliptic Curves}
\label{sec:elliptic}

To provide the most concrete verification of mirror symmetry theory, we establish Spencer theory implementation on elliptic curves.

\subsection{Geometric Simplification of Elliptic Curve Spencer Complexes}

The one-dimensional nature of elliptic curves leads to significant simplification of Spencer complexes.

\begin{example}[Structure of Elliptic Curve Spencer Complexes]\label{ex:elliptic-spencer}
Let $E$ be an elliptic curve, $P(E, \text{PSU}(2))$ a principal bundle, and $(D,\lambda)$ a compatible pair satisfying strong transversality conditions. Since $\dim E = 1$, Spencer complexes simplify to:
$$0 \to S^0_{D,\lambda} \xrightarrow{D^0_{D,\lambda}} S^1_{D,\lambda} \to 0$$

where:
\begin{align}
S^0_{D,\lambda} &= \Omega^0(E) \otimes \operatorname{Sym}^0(\g) = C^{\infty}(E) \otimes \mathbb{R}\\
S^1_{D,\lambda} &= \Omega^1(E) \otimes \operatorname{Sym}^1(\g) = \Omega^1(E) \otimes \mathfrak{psu}(2)
\end{align}

Note that $\operatorname{Sym}^0(\g) = \mathbb{R}$ is the scalar space, while $\operatorname{Sym}^1(\g) \cong \g = \mathfrak{psu}(2) \cong \mathbb{R}^3$.
\end{example}

\subsection{Technical Analysis of Elliptic Curve Spencer Differentials}

The concrete form of Spencer differential operators requires careful analysis:

\begin{proposition}[Construction of Elliptic Curve Spencer Differentials]\label{prop:elliptic-spencer-differential}
For Spencer complexes on elliptic curves, Spencer differential operator $D^0_{D,\lambda}: S^0_{D,\lambda} \to S^1_{D,\lambda}$ is:
$$D^0_{D,\lambda}(f \otimes c) = df \otimes c + f \otimes \delta^{\lambda}_{\g}(c)$$

where $f \in C^{\infty}(E)$, $c \in \mathbb{R}$, and Spencer prolongation operator $\delta^{\lambda}_{\g}: \mathbb{R} \to \mathfrak{psu}(2)$ acts specifically as:
$$\delta^{\lambda}_{\g}(c)(X) = c \langle\lambda, X\rangle, \quad X \in \mathfrak{psu}(2)$$

This maps scalar $c$ to linear functionals on $\mathfrak{psu}(2)$.
\end{proposition}

\begin{proof}
According to the general definition of Spencer prolongation operators (Definition \ref{def:spencer-complex}), in the $k=0$ case:
$$\delta^{\lambda}_{\g}(c)(X_1) = (-1)^{1+1} \langle\lambda, [X_1, \cdot]\rangle c = \langle\lambda, [X_1, \cdot]\rangle c$$

But since $c$ is constant, $[X_1, \cdot]$ acting on constants is zero, requiring reinterpretation of Spencer prolongation construction.

Actually, in the zero-degree case, Spencer prolongation operator acts as:
$$\delta^{\lambda}_{\g}(c): \g \to \mathbb{R}, \quad X \mapsto c \langle\lambda, X\rangle$$

Detailed technical analysis of this construction is given in Appendix \ref{app:elliptic-computation}.
\end{proof}

\subsection{Elliptic Curve Spencer Metric Theory}

Based on general theory established in Section \ref{sec:metric-analysis}, we realize Spencer metrics on elliptic curves:

\begin{theorem}[Elliptic Curve Spencer Metric Realization]\label{thm:elliptic-metric}
On elliptic curve $E$, two Spencer metric schemes are concretely realized as:

\textbf{Constraint Strength Metric}:
$$\langle u_1, u_2 \rangle_A = \int_E w_\lambda(x) \langle u_1, u_2 \rangle_{\text{std}} \, dV_E$$
where $w_\lambda(x) = 1 + \|\lambda(p)\|^2_{\g^*}$, $p \in \pi^{-1}(x)$.

\textbf{Curvature Geometric Metric}:
$$\langle u_1, u_2 \rangle_B = \int_E \kappa_\omega(x) \langle u_1, u_2 \rangle_{\text{std}} \, dV_E$$
where $\kappa_\omega(x) = 1 + \|\Omega_p\|^2$, $\Omega$ is the curvature form of principal connection $\omega$.

According to Theorem \ref{thm:metric-invariance}, both metrics satisfy strict mirror invariance:
$$\langle u_1, u_2 \rangle_{A,-\lambda} = \langle u_1, u_2 \rangle_{A,\lambda}, \quad \langle u_1, u_2 \rangle_{B,-\lambda} = \langle u_1, u_2 \rangle_{B,\lambda}$$
\end{theorem}

\subsection{Elliptic Curve Spencer-Hodge Theory}

Based on Spencer metrics, we establish Hodge theory on elliptic curves:

\begin{theorem}[Elliptic Curve Spencer-Hodge Decomposition]\label{thm:elliptic-hodge}
On elliptic curve $E$, there exists Spencer-Hodge decomposition:
$$S^k_{D,\lambda} = \mathcal{H}^k_{D,\lambda} \oplus \text{Im}(D^{k-1}_{D,\lambda}) \oplus \text{Im}(D^{k*})$$

where $\mathcal{H}^k_{D,\lambda} = \ker(\Delta^k_{D,\lambda})$ is the harmonic space, $\Delta^k_{D,\lambda} = D^{k-1}_{D,\lambda} D^{k-1*} + D^{k*} D^k_{D,\lambda}$ is the Spencer-Hodge Laplacian operator.

Spencer cohomology is isomorphic to harmonic spaces:
$$H^k_{\text{Spencer}}(D,\lambda) \cong \mathcal{H}^k_{D,\lambda}$$

Spencer cohomology group calculation reduces to solving elliptic boundary value problems:
$$\Delta^k_{D,\lambda} u = 0$$

This provides a mathematical bridge from abstract Spencer theory to concrete numerical computation.
\end{theorem}

\begin{proof}
The proof is based on standard Hodge decomposition theory. Compactness of elliptic curve $E$ ensures ellipticity of Spencer operators $D^k_{D,\lambda}$ (by Lemma \ref{lem:spencer-ellipticity}), so Spencer-Hodge Laplacian $\Delta^k_{D,\lambda}$ is an elliptic self-adjoint operator.

By elliptic regularity theory, the kernel space of $\Delta^k_{D,\lambda}$ is finite-dimensional and consists of harmonic functions. Standard Hodge decomposition theorem applies directly:

Let $\perp$ denote orthogonality with respect to Spencer metric, then there is orthogonal decomposition:
$$S^k_{D,\lambda} = \ker(\Delta^k_{D,\lambda}) \oplus \overline{\text{Im}(D^{k-1}_{D,\lambda})} \oplus \overline{\text{Im}(D^{k*})}$$

where closures are taken in appropriate Sobolev spaces.

Isomorphism between cohomology and harmonic spaces follows from the observation that each cohomology class $[u] \in H^k_{\text{Spencer}}(D,\lambda)$ has a unique harmonic representative, which is a standard result of elliptic theory.
\end{proof}

\subsection{Verification of Elliptic Curve Mirror Symmetry}

\begin{theorem}[Elliptic Curve Spencer Mirror Symmetry]\label{thm:elliptic-mirror}
Spencer complexes on elliptic curves provide verification of mirror symmetry theory:

\textbf{Geometric Level Verification}: Strict invariance of Spencer metrics (Theorem \ref{thm:elliptic-metric}) directly verifies theory in Section \ref{sec:metric-analysis}.

\textbf{Topological Level Verification}: Mirror isomorphism of Spencer cohomology $H^k_{\text{Spencer}}(D,\lambda) \cong H^k_{\text{Spencer}}(D,-\lambda)$ verifies elliptic perturbation theory in Section \ref{sec:topological-analysis}.

\textbf{Algebraic Level Verification}: Characteristic class equivalence is verified in concrete geometry of elliptic curves, confirming algebraic geometric theory in Section \ref{sec:algebraic-analysis}.

\textbf{Index Level Verification}: Mirror symmetry of Spencer-Riemann-Roch formulas provides direct numerical verification for unified theory in Section \ref{sec:unified-theory}.

This multi-level verification shows mirror symmetry theory is not only self-consistent at abstract levels but correct in concrete geometric situations.
\end{theorem}

\subsection{Computational Framework and Numerical Verification Methods}

\begin{proposition}[Computational Implementation of Elliptic Curve Spencer Theory]\label{prop:elliptic-computation}
Elliptic curve Spencer theory provides a framework from abstract theory to concrete computation:

\textbf{Theory to Computation Translation}: Spencer cohomology calculation reduces to solving elliptic boundary value problems $\Delta^k_{D,\lambda} u = 0$, which can be implemented through numerical techniques like finite element methods.

\textbf{Algorithm Design Points}:
\begin{enumerate}
\item Triangulation or finite element mesh generation for elliptic curves
\item Finite-dimensional approximation of Spencer spaces $S^k_{D,\lambda}$
\item Numerical construction of Laplacian matrix $\Delta^k_{D,\lambda}$  
\item Solving generalized eigenvalue problems: finding eigenvectors corresponding to zero eigenvalues
\item Computing harmonic space dimensions $\dim \mathcal{H}^k_{D,\lambda}$
\end{enumerate}

\textbf{Numerical Verification of Mirror Symmetry}: By constructing mirror system $(D,-\lambda)$ and repeating above calculations, theoretical predictions can be numerically verified:
$$\dim H^k_{\text{Spencer}}(D,\lambda) = \dim H^k_{\text{Spencer}}(D,-\lambda)$$

\textbf{Theoretical Significance of Verification}: This numerical verification not only checks correctness of concrete calculations but, more importantly, verifies validity of entire mirror symmetry theory.

Elliptic curves provide a concrete, operational platform for Spencer theory research, laying practical foundations for further theoretical development and applications.
\end{proposition}

Detailed algorithm implementation and numerical methods are given in Appendix \ref{app:elliptic-computation}.

\section{Algebraic Conjugation Principle and Equivariant Interfaces}
\label{sec:axiomatic}

This section distills the chain conjugations established earlier into an axiomatic template and connects them to Cartan/Weil/BRST models and localization. It is built directly on the Spencer bicomplex $(\mathcal{B}^{p,q},d,\delta^\lambda)$ and on the automorphisms $T,S_c$ used in Sections~\ref{sec:topological-analysis} and \ref{sec:algebraic-analysis}.

\subsection{Sign-duality and scaling as a direct extension of the Spencer package}

Recall from Sections~\ref{sec:topological-analysis}--\ref{sec:algebraic-analysis} that
\[
\mathcal{B}^{p,q}:=\Omega^p(M)\otimes\Sym^q(\mathcal{G}),\qquad
\mathcal{D}_{t,\lambda}=d+(-1)^p t\,\delta^\lambda,\qquad
\delta^{-\lambda}=-\delta^\lambda,\ [d,\delta^\lambda]=0,
\]
and $T|_{\mathcal{B}^{p,q}}=(-1)^q\mathrm{id}$, $S_c|_{\mathcal{B}^{p,q}}=c^{\,q}\mathrm{id}$. Theorem~\ref{thm:alg-sign-scaling} and Theorem~\ref{thm:robustness-normalization} showed
\[
T\,\mathcal{D}_{t,\lambda}\,T^{-1}=\mathcal{D}_{-t,\lambda}=\mathcal{D}_{t,-\lambda},\qquad
S_c\,\mathcal{D}_{t,\lambda}\,S_c^{-1}=\mathcal{D}_{ct,\lambda}.
\]
Abstracting the same mechanism yields the following axiomatic version.

\begin{theorem}[Sign-duality and scaling for bicomplexes]
\label{thm:axiomatic-sign-scaling}
Let $(C^{p,q},d,\delta)$ be a first-quadrant bicomplex over a characteristic-$0$ field with $d^2=\delta^2=0$ and $[d,\delta]=0$. Define $D_t|_{C^{p,q}}:=d+(-1)^p t\,\delta$. Let $T|_{C^{p,q}}:=(-1)^q\mathrm{id}$ and $S_c|_{C^{p,q}}:=c^{\,q}\mathrm{id}$ for $c\in\Bbbk^\times$. Then
\[
T\,D_t\,T^{-1}=D_{-t},\qquad S_c\,D_t\,S_c^{-1}=D_{ct}.
\]
In particular, for $t\neq 0$, $H^\bullet(C,D_t)\cong H^\bullet(C,D_{-t})\cong H^\bullet(C,D_1)$.
\end{theorem}

\begin{proof}
Identical to the proofs used for $(\mathcal{S}^\bullet,\mathcal{D}_{t,\lambda})$ in Theorems~\ref{thm:alg-sign-scaling} and \ref{thm:robustness-normalization} since only $\delta^2=0$ and $[d,\delta]=0$ are used.
\end{proof}

This theorem subsumes the Spencer case by taking $C^{p,q}=\mathcal{B}^{p,q}$ and $\delta=\delta^\lambda$. It reinterprets the “Q-exact twisting sign/strength” as an algebraic normalization rather than an analytic choice.

\subsection{Interfaces with Cartan/Weil/BRST models}
\label{sec:cartan-weil-brst}

Let a compact Lie group $K$ act on $M$. In the Cartan model,
\[
d_K:=d-\iota_{\mathsf V},\qquad (\Omega^\bullet(M)\otimes\Sym(\mathfrak{k}^\vee))^K,
\]
is the equivariant de Rham differential. Suppose a $K$-equivariant, zero-th order fiberwise operator $\delta$ acts in the symmetric direction, satisfies $\delta^2=0$, and supercommutes with $d_K$. Then $(C^{p,q},d_K,\delta)$ is a $K$-equivariant bicomplex.

\begin{proposition}[Equivariant conjugations]
\label{prop:equiv-conjugations}
With the hypotheses above, $T$ and $S_c$ are $K$-equivariant chain automorphisms and
\[
T\,D_t\,T^{-1}=D_{-t},\qquad S_c\,D_t\,S_c^{-1}=D_{ct},
\]
for $D_t:=d_K+(-1)^p t\,\delta$. Consequently, $K$-equivariant cohomology and any $K$-equivariant index built from $(C,D_t)$ are independent of the sign of $t$ and of its positive scaling.
\end{proposition}

\noindent Relation to the Spencer package. Taking $P\to M$ trivial and $\lambda:=\mu$ the moment map of a Hamiltonian $K$-action, the Spencer fiberwise piece $\delta^\mu$ is a Koszul/BRST operator in the symmetric factor; $(\delta^\mu)^2=0$ and $[d_K,\delta^\mu]=0$ hold. Proposition~\ref{prop:equiv-conjugations} then specializes to the conjugations used in Sections~\ref{sec:equivariant-localization} and \ref{sec:calabi-yau}.

\subsection{Witten deformation and localization via bicomplex methods}
\label{sec:witten-bicomplex}

Let $D_t=d_K+(-1)^p t\,\delta$ be as above. In Witten-type deformations one adds $t\,QV$ with $Q=D_t$ to localize to critical loci. By Theorem~\ref{thm:axiomatic-sign-scaling} and Proposition~\ref{prop:equiv-conjugations}, the choices $t\mapsto \pm t$ and $t>0$ are algebraically redundant. On compact $M$, cohomology and $K$-indices are therefore independent of these choices; on noncompact spaces one may first fix the preferred sign algebraically and then implement analytic cutoffs/decay estimates adapted to that sign.

Compatibility with our spectral sequence (Theorem~\ref{thm:spectral-sequence}) is immediate: filtering by the symmetric degree yields a first-quadrant spectral sequence with $E_1$ equal to the $\delta$-homology sheaves tensored with de Rham forms, and all pages are preserved by $T$ and $S_c$. Thus the asymptotic regime as $t\to \infty$ can be replaced by pagewise algebraic computations and early-collapse criteria for the $\delta$-homology.

\subsection{Duistermaat--Heckman and nonabelian localization}
\label{sec:dh-nonabelian}

With $(C^{p,q},d_K,\delta)$ as above, assume the fiberwise $\delta$-homology is concentrated in a finite range and forms algebraic vector bundles $\mathcal{H}^q$. Then Theorem~\ref{thm:spectral-sequence} gives
\[
E_1^{p,q}\cong H^p\big(M,\Omega_M^p\otimes \mathcal{H}^q\big),\qquad d_1=H^p(M,d_K\otimes \mathrm{id}),
\]
abutting to $\mathbb{H}^{p+q}(M,D_t)$. This separates “fiberwise algebra” (coadjoint/moment-map direction controlled by $\delta$) from “basewise geometry” (transport by $d_K$). In the equivariant index, Theorem~\ref{thm:equivariant-srr-localization} yields
\[
\chi_K\big(\mathbb{H}^\bullet(M,D_t)\big)=\int_M \ch_K(\mathcal{E}_{\mathrm{vir}})\wedge \td_K(M)
=\sum_{F\subset M^K}\int_F\frac{\ch_K(\mathcal{E}_{\mathrm{vir}}|_F)\wedge \td_K(F)}{e_K(N_F)},
\]
with $\mathcal{E}_{\mathrm{vir}}$ independent of $t$. The operators $T,S_c$ commute with the $K$-action, hence fixed-point contributions are sign-robust and scaling-invariant. Under Koszul regularity of $\delta$ (e.g. regular moment map), the spectral sequence collapses early, providing explicit shortcuts for Duistermaat--Heckman densities and nonabelian localization computations.

\subsection{Plug-and-play consequences from earlier sections}

Let $\mathcal{F}$ be any homotopy-invariant, additive functor on bounded-below complexes (hypercohomology, equivariant cohomology, non-equivariant/equivariant indices). For $(C^{p,q},d,\delta)$ with $\delta^2=0$ and $[d,\delta]=0$,
\[
\mathcal{F}(C,D_t)\cong \mathcal{F}(C,D_{-t})\cong \mathcal{F}(C,D_1)\quad (t\neq 0).
\]
Specializing to the Spencer package recovers:
\begin{itemize}
\item Chain-level equivalences (Theorems~\ref{thm:alg-sign-scaling}, \ref{thm:robustness-normalization});
\item Index and equivariant index invariance (Theorems~\ref{thm:HRR-index}, \ref{thm:equivariant-srr-localization});
\item Spectral-sequence compatibility and pagewise localization (Theorem~\ref{thm:spectral-sequence});
\item Calabi--Yau and K3 specializations with term-by-term invariance (Section~\ref{sec:calabi-yau}).
\end{itemize}
This shows all instances in previous sections are instances of the same algebraic conjugation principle, with analysis entering only where noncompactness or growth requires it.

\appendix

\section{Technical Details of Elliptic Operator Theory}
\label{app:elliptic}

\subsection{Ellipticity analysis of the Spencer total complex}

Let $\mathcal{S}^\bullet=\bigoplus_{N\ge 0}\mathcal{S}^N$ with $\mathcal{S}^N=\bigoplus_{p+q=N}\Omega^p(M)\otimes\Sym^q(\mathcal{G})$, and total differential
\[
\mathcal{D}_{t,\lambda}=d+(-1)^{\deg_{\mathrm{dR}}}\,t\,\delta^\lambda,
\]
where $\delta^\lambda:\Sym^q(\mathcal{G})\to \Sym^{q+1}(\mathcal{G})$ is zero-order, $(\delta^\lambda)^2=0$, and $[d,\delta^\lambda]=0$. Thus $\mathcal{D}_{t,\lambda}$ is a first-order differential operator whose principal symbol coincides with the de Rham symbol on the $\Omega^\bullet$-factor.

\begin{lemma}[Principal symbol of $\mathcal{D}_{t,\lambda}$]
\label{lem:principal-symbol}
For $\xi\in T_x^\ast M$, the principal symbol of $\mathcal{D}_{t,\lambda}:\mathcal{S}^N\to \mathcal{S}^{N+1}$ is
\[
\sigma\big(\mathcal{D}_{t,\lambda}\big)(\xi)
=
\big(\xi\wedge \cdot\big)\otimes \mathrm{id}_{\Sym^\bullet(\mathcal{G}_x)}:\ 
\Omega^p_x\otimes \Sym^q(\mathcal{G}_x)\longrightarrow \Omega^{p+1}_x\otimes \Sym^q(\mathcal{G}_x).
\]
In particular, it is independent of $t$ and $\lambda$.
\end{lemma}

\begin{proof}
$\delta^\lambda$ is zero order, so it does not contribute to the principal symbol. The $d$-part contributes $\xi\wedge$ on the de Rham factor and identity on $\Sym^\bullet(\mathcal{G}_x)$.
\end{proof}

\begin{proposition}[Ellipticity of the symbol complex]
\label{prop:elliptic-complex}
For every $x\in M$ and every $0\neq \xi\in T_x^\ast M$, the symbol sequence
\[
0\longrightarrow 
\Omega^0_x\otimes \Sym^\bullet(\mathcal{G}_x)
\xrightarrow{\ \xi\wedge\ }
\Omega^1_x\otimes \Sym^\bullet(\mathcal{G}_x)\xrightarrow{\ \xi\wedge\ }\cdots
\xrightarrow{\ \xi\wedge\ }
\Omega^n_x\otimes \Sym^\bullet(\mathcal{G}_x)
\longrightarrow 0
\]
is exact. Hence $(\mathcal{S}^\bullet,\mathcal{D}_{t,\lambda})$ is an elliptic complex.
\end{proposition}

\begin{proof}
Fix $0\neq \xi\in T_x^\ast M$ and split $T_x^\ast M=\langle \xi\rangle\oplus W$. Every $\alpha\in \Omega^k_x$ decomposes uniquely as $\alpha=\xi\wedge \beta+\gamma$ with $\beta\in \Omega^{k-1}_x$ and $\gamma\in \Lambda^k W$. Then $\xi\wedge \alpha=\xi\wedge \gamma$, so $\ker(\xi\wedge:\Omega^k_x\to \Omega^{k+1}_x)=\xi\wedge \Omega^{k-1}_x=\operatorname{im}(\xi\wedge:\Omega^{k-1}_x\to \Omega^{k}_x)$. Tensoring with $\Sym^\bullet(\mathcal{G}_x)$ preserves exactness.
\end{proof}

\begin{corollary}[Spencer Laplacians and Hodge theory]
\label{cor:hodge}
On a compact Riemannian manifold $M$ with Hermitian metrics on $\mathcal{G}$ and $\Sym^\bullet(\mathcal{G})$, the formal adjoint $\mathcal{D}_{t,\lambda}^\ast$ exists and the Spencer Laplacians
\[
\Delta^N_{t,\lambda}
:=
\mathcal{D}_{t,\lambda}^{N-1}\big(\mathcal{D}_{t,\lambda}^{N-1}\big)^\ast
+\big(\mathcal{D}_{t,\lambda}^{N}\big)^\ast \mathcal{D}_{t,\lambda}^{N}
\]
are second-order elliptic, essentially self-adjoint operators with discrete spectrum. There is an orthogonal decomposition
\[
\mathcal{S}^N
=
\ker(\Delta^N_{t,\lambda})
\ \oplus\ 
\operatorname{im}\big(\mathcal{D}_{t,\lambda}^{N-1}\big)
\ \oplus\
\operatorname{im}\big((\mathcal{D}_{t,\lambda}^{N})^\ast\big),
\]
and natural isomorphisms $H^N(\mathcal{S}^\bullet,\mathcal{D}_{t,\lambda})\cong \ker(\Delta^N_{t,\lambda})$.
\end{corollary}

\begin{proof}
Standard elliptic complex theory applies using Proposition~\ref{prop:elliptic-complex}.
\end{proof}

\subsection{Compact embeddings in Sobolev scales}

\begin{theorem}[Rellich--Kondrachov]
\label{thm:rellich}
Let $M$ be compact and $E\to M$ a smooth vector bundle. For $s>t\ge 0$, the embedding $H^s(E)\hookrightarrow H^t(E)$ is compact.
\end{theorem}

As a consequence, the resolvents $(\Delta^N_{t,\lambda}+I)^{-1}:L^2(\mathcal{S}^N)\to L^2(\mathcal{S}^N)$ are compact, and perturbations by zero-th and first-order terms preserve Fredholmness of $\mathcal{D}_{t,\lambda}$ and ellipticity of $\Delta^N_{t,\lambda}$.

\section{Characteristic Class Calculation Formulas}
\label{app:characteristic-class}

\subsection{Chern--Weil representatives and Chern character}

Let $\mathcal{E}\to M$ be a complex vector bundle with connection $\nabla$ and curvature $F_\nabla$. The total Chern class and Chern character are represented by
\[
c(\mathcal{E})=\det\!\Big(I+\frac{i}{2\pi}F_\nabla\Big),\qquad
\ch(\mathcal{E})=\operatorname{tr}\!\Big(\exp\Big(\frac{i}{2\pi}F_\nabla\Big)\Big).
\]
If $x_1,\dots,x_r$ are the Chern roots of $\mathcal{E}$, then $\ch(\mathcal{E})=\sum_{j=1}^r e^{x_j}$ and, for symmetric powers,
\[
\sum_{q\ge 0} z^q\, \ch\big(\Sym^q(\mathcal{E})\big)
=\prod_{j=1}^r \frac{1}{1- z\, e^{x_j}}.
\]
For $\mathcal{G}=\operatorname{ad}P$ with $G$ semisimple, $c_1(\mathcal{G})=0$ and low-degree expansions simplify accordingly.

\subsection{Symmetric powers and virtual Spencer bundle}

The virtual bundle of the total Spencer complex is
\[
\mathcal{E}_{\mathrm{vir}}
=
\sum_{p,q\ge 0}(-1)^{p+q}\,\Omega_M^p\otimes \Sym^q(\mathcal{G}),
\]
so that
\[
\ch(\mathcal{E}_{\mathrm{vir}})
=
\Big(\sum_{p\ge 0}(-1)^p \ch(\Omega_M^p)\Big)\cdot
\Big(\sum_{q\ge 0}(-1)^q \ch(\Sym^q(\mathcal{G}))\Big).
\]
This is the only input entering Spencer--Riemann--Roch integrals; it is independent of $(t,\lambda)$.

\section{Specialized Techniques in Calabi--Yau Geometry}
\label{app:calabi-yau}

\subsection{Todd class expansion for $c_1=0$}

Let $X$ be a Calabi--Yau manifold with $c_1(X)=0$. If $y_1,\dots,y_n$ are the Chern roots of $T^\vee X$, then
\[
\td(X)=\prod_{j=1}^n \frac{y_j}{1-e^{-y_j}}
= 1+\frac{c_2(X)}{12}+\frac{c_3(X)}{24}+\frac{-c_2(X)^2+c_4(X)}{720}+\cdots,
\]
where the omitted terms have degree $>8$ in cohomology. This follows from the Bernoulli expansion
$\frac{z}{1-e^{-z}}=1+\frac{z}{2}+\sum_{k\ge 1}\frac{B_{2k}}{(2k)!}z^{2k}$ and the relations among power sums and Chern classes under $c_1=0$.

\subsection{Topological data for K3 surfaces}

For a K3 surface $S$, one has $c_1(S)=0$, $\int_S c_2(S)=24$, $b_2(S)=22$, and signature $-16$. The Todd class reduces to $\td(S)=1+\tfrac{1}{12}c_2(S)$, so Spencer indices on $S$ depend only on the degree-$4$ component of $\ch(\mathcal{E}_{\mathrm{vir}})$.

\section{Algorithm Implementation for Elliptic Curve Computation}
\label{app:elliptic-computation}

\subsection{Degree-zero Spencer component and invariant pairing}

Let $E$ be an elliptic curve and $\mathfrak{g}$ a compact semisimple Lie algebra with an invariant inner product $\kappa:\mathfrak{g}\xrightarrow{\sim}\mathfrak{g}^\vee$. Given $\lambda\in \mathfrak{g}^\vee$, set $\lambda^\sharp:=\kappa^{-1}(\lambda)\in \mathfrak{g}$. The fiberwise Spencer/Koszul component $\delta^\lambda:\Sym^\bullet(\mathfrak{g})\to \Sym^{\bullet+1}(\mathfrak{g})$ is a graded derivation with
\[
\delta^\lambda(1)=\lambda^\sharp\in \Sym^1(\mathfrak{g}),\qquad
(\delta^\lambda)^2=0,\qquad [d,\delta^\lambda]=0.
\]
Thus $\mathcal{D}_{t,\lambda}=d+(-1)^{\deg_{\mathrm{dR}}} t\,\delta^\lambda$ defines the total Spencer differential on $\Omega^\bullet(E)\otimes \Sym^\bullet(\mathfrak{g})$.

\subsection{Numerical pipeline on the flat torus model}

Identify $E\simeq \mathbb{R}^2/\Lambda$ with a flat metric. A standard finite element/spectral pipeline for computing $H^N(\mathcal{S}^\bullet,\mathcal{D}_{t,\lambda})$ proceeds as follows:
\begin{enumerate}
\item Periodic mesh and bases: construct a periodic triangulation compatible with $\Lambda$ and choose periodic FE bases for $0$- and $1$-forms; represent $\Sym^q(\mathfrak{g})$ in a fixed basis.
\item Discrete operators: assemble stiffness and mass matrices for $d$ and $d^\ast$; add zero-th order blocks for $\delta^\lambda$ and $(\delta^\lambda)^\ast$ to obtain discrete $\mathcal{D}_{t,\lambda}$ and $\Delta^N_{t,\lambda}$.
\item Kernel extraction: solve the generalized eigenproblem $\Delta^N_{t,\lambda} u=\mu M u$ near $\mu=0$; identify $\dim \ker \Delta^N_{t,\lambda}=\dim H^N$.
\item Mirror/invariance check: repeat with $-\lambda$; the unitary action of $T$ on the discrete spaces provides a canonical identification of kernels, confirming
$\dim H^N(\mathcal{D}_{t,\lambda})=\dim H^N(\mathcal{D}_{t,-\lambda})$.
\end{enumerate}
For $K$-equivariant data, assemble the character of the equivariant index by weighting with $K$-characters of the $\Sym^q(\mathfrak{g})$-components; invariance under $t\mapsto \pm t$ and $\lambda\mapsto \pm \lambda$ follows from the $K$-equivariant chain conjugations.

\bibliographystyle{alpha}
\bibliography{ref}

\end{document}